\documentclass[11pt, reqno]{amsart}
\usepackage{amssymb}
\usepackage{verbatim}
\usepackage[vcentermath]{youngtab}
\usepackage{float}
\usepackage{hyperref}
\usepackage{tikz}
\usepackage{hyperref}
\usepackage[left=1.1in, right=1.1in, top=1in, bottom=.9in]{geometry}

\usepackage{times}
\usepackage[T1]{fontenc}
\usepackage{mathrsfs}
\usepackage{epsfig}
\usepackage{color}
\usepackage{array} 
\newcolumntype{L}{>{$}l<{$}}
\newcolumntype{R}{>{$}r<{$}}

\usepackage{enumitem}

\newtheorem{thm}{Theorem}[section]
\newtheorem{lemma}[thm]{Lemma}

\newtheorem{prop}[thm]{Proposition}

\theoremstyle{remark}

\newtheorem*{remark}{Remark}

\theoremstyle{definition}
\newtheorem{defn}[thm]{Definition}

\newtheorem{example}[thm]{Example}

\def\ZZ{\mathbb{Z}}
\def\QQ{\mathbb{Q}}
\def\PP{\mathbb{P}}
\def\NN{\mathbb{N}}

\def\AA{\mathbb{A}}

\def\FF{\mathbb{F}}

\def\RR{\mathbb{R}}
\def\CC{\mathbb{C}}

\def\multi#1#2{\ensuremath{\left(\kern-.3em\left(\genfrac{}{}{0pt}{}{#1}{#2}\right)\kern-.3em\right)}}

\newcommand{\Sym}{\mathrm{Sym}}
\newcommand{\poly}{\mathrm{Poly}}

\newcommand{\irr}{\mathrm{Irr}}

\numberwithin{equation}{section}

\begin{document}

\title[Euler characteristic of $\irr_{d,n}(\RR)$]{Euler characteristic of the space of\\ real multivariate irreducible polynomials}

\author{Trevor Hyde}
\address{Dept. of Mathematics\\
University of Chicago \\
Chicago, IL 60637\\
}
\email{tghyde@uchicago.edu}


\maketitle

Given a field $K$ and integers $d, n \geq 1$, let $\poly_{d,n}(K)$ denote the set of all monic total degree $d$ polynomials in $n$ variables with coefficients in $K$ (see Definition \ref{def poly}.) Let $\irr_{d,n}(K) \subseteq \poly_{d,n}(K)$ denote the subset of all polynomials which are irreducible over $K$. Our main result expresses the compactly supported Euler characteristic of $\irr_{d,n}(\RR)$ in terms of the so-called balanced binary expansion of the number of variables $n$. The \textit{balanced binary expansion} of an integer $n\geq 1$ is the unique expression
\[
    n = \sum_{k=0}^\ell b_k 2^k = 2^{k_{2m}} - 2^{k_{2m-1}} + 2^{k_{2m-2}} - \ldots + 2^{k_{1}} - 2^{k_0}
\]
where $0 \leq k_0 < k_1 < \ldots < k_{2m}$ is a strictly increasing sequence of natural numbers of even length and the signs on the right hand side alternate.

\begin{thm}
\label{thm main intro}
Let $n\geq 1$ be an integer with balanced binary expansion $n = \sum_{k=0}^\ell b_k 2^k$ and let $\chi_c$ denote the compactly supported Euler characteristic. Then
\[
    \chi_c(\irr_{d,n}(\RR)) = \begin{cases} b_k & \text{if } d = 2^k,\\ 0 & \text{otherwise.}\end{cases}
\]
\end{thm}

\begin{example}
The balanced binary expansion of $n = 13$ is
\[
    13 = 2^4 - 2^2 + 2^1 - 1. 
\]
Thus Theorem \ref{thm main intro} implies that
\[
    \chi_c(\irr_{d,13}(\RR)) = \begin{cases}\hspace{.75em}1 & \text{if } d = 2 \text{ or } 16,\\ -1 & \text{if } d = 1 \text{ or }4,\\ \hspace{.75em}0 & \text{otherwise.} \end{cases}
\]
That is, the space of degree $d$ irreducible polynomials in 13 variables with coefficients in $\RR$ has vanishing compactly supported Euler characteristic in all degrees except $d = 1, 2, 4, 16$, and these exceptional degrees are determined by the binary expansion of $n = 13$.
\end{example}

\begin{example}
For univariate polynomials ($n = 1$) we can explicitly construct cell decompositions of $\irr_{d,1}(\RR)$ for all $d\geq 1$. In particular, the fundamental theorem of algebra and the quadratic formula imply
\begin{align*}
    \irr_{1,1}(\RR) &= \{x + a : a \in \RR\} \cong \RR,\\
    \irr_{2,1}(\RR) &= \{x^2 + bx + c : b^2 < 4c\} \cong \RR^2,\\
    \irr_{d,1}(\RR) &= \emptyset \text{ for $d > 2$.}
\end{align*}
Since $\chi_c(\RR^k) = (-1)^k$ it follows that
\[
    \chi_c(\irr_{d,1}(\RR)) = \begin{cases} -1 & \text{if }d = 1,\\ \hspace{.75em}1 & \text{if }d = 2,\\ \hspace{.75em}0 & \text{if }d > 2. \end{cases}
\]
This calculation is consistent with Theorem \ref{thm main intro} since $1 = 2 - 1$ is the balanced binary expansion of $1$.
\end{example}

The connection between the Euler characteristic of $\irr_{d,n}(\RR)$ and the binary expansion of the number of variables $n$ was discovered empirically and came as a surprise to the author. This result suggests that the spaces $\irr_{d,n}(\RR)$ may have interesting cell decompositions determined by the additive structure of $n$. Our proof of Theorem \ref{thm main intro} arrives at this Euler characteristic computation indirectly via generating functions; the dependence on the binary expansion of $n$ passes through the well-known theorem of Lucas on the mod $p$ residue of binomial coefficients (see Lemma \ref{lem Lucas}.)

Chen \cite{chen} proved that the singular and compactly supported cohomology of the spaces $\irr_{d,n}(\CC)$ stabilizes in certain cohomological degree regimes as $n \rightarrow \infty$. We do not know whether Chen's methods can be adapted to analyze the compactly supported cohomology of $\irr_{d,n}(\RR)$.

In \cite{hyde_liminal} the author introduced the \textit{higher necklace polynomials} $M_{d,n}(x) \in \QQ[x]$ which interpolate the point counts of $\irr_{d,n}(\FF_q)$ for finite fields $\FF_q$ where $q$ is a prime power,
\[
    |\irr_{d,n}(\FF_q)| = M_{d,n}(q).
\]
En route to proving Theorem \ref{thm main intro} we show that the compactly supported Euler characteristics of $\irr_{d,n}(\RR)$ and $\irr_{d,n}(\CC)$ are realized as values of the higher necklace polynomials.

\begin{thm}
\label{thm eval intro}
For all $d, n \geq 1$,
\[
    \chi_c(\irr_{d,n}(\RR)) = M_{d,n}(-1) \hspace{.75in} \chi_c(\irr_{d,n}(\CC)) = M_{d,n}(1) = \begin{cases} n & \text{if }d = 1\\ 0 & \text{otherwise.}\end{cases}
\]
\end{thm}

The compactly supported Euler characteristic may be viewed as a topological extension of the cardinality of a finite set. Since $\chi_c(\RR) = -1$ and $\chi_c(\CC) = 1$, it is sometimes said that $\RR$ and $\CC$ behave like ``fields with $-1$ and $1$ element'' respectively; Theorem \ref{thm eval intro} provides an example of a common situation where this heuristic applies. The strategy behind Theorem \ref{thm eval intro} and similar results is to find some universal decomposition of $\irr_{d,n}(K)$ into a disjoint union of affine cells $K^m$. Provided the decomposition is sufficiently nice, algebraic properties of point counts and compactly supported Euler characteristics translate these universal cell decompositions into a common polynomial expression for the respective invariants.

With Theorem \ref{thm eval intro} in hand, Theorem \ref{thm main intro} reduces to evaluating the polynomial $M_{d,n}(x)$ at $x = -1$. More generally we prove the following result on evaluations $M_{d,n}(\zeta_p)$ of the higher necklace polynomials at prime order roots of unity. Note that a \textit{balanced base $p$ expansion} of an integer $n$ is an expression
\[
    n = \sum_{k=0}^\ell b_k p^k = p^{k_{2m}} - p^{k_{2m-1}} + p^{k_{2m-2}} - \ldots + p^{k_{1}} - p^{k_0}
\]
where $0 \leq k_0 < k_1 < \ldots < k_{2m}$ is a strictly increasing sequence of natural numbers of even length and the signs on the right hand side alternate (see Section \ref{sec neck eval}.)

\begin{thm}
\label{thm neck eval intro}
Let $p$ be a prime, let $d, n \geq 1$ be integers, and suppose that $n$ has a balanced base $p$ expansion $n = \sum_{k=0}^\ell b_k p^k$. If $\zeta_p$ is a primitive $p$th root of unity, then
\[
    M_{d,n}(\zeta_p) = \begin{cases} b_k & \text{if }d = p^k\\ 0 & \text{otherwise.}\end{cases}
\]
\end{thm}

In particular, for a fixed $n$ with balanced base $p$ expansion, $M_{d,n}(x)$ is divisible by the $p$th cyclotomic polynomial $\Phi_p(x)$ for all but finitely many $d$. When $n = 1$ the higher necklace polynomials specialize to the classic sequence $M_d(x) := M_{d,1}(x)$ of \textit{necklace polynomials}. Odesky and the author \cite{hyde odesky} observed that the necklace polynomials $M_d(x)$ are highly reducible over $\QQ$ with nearly all of their irreducible factors being cyclotomic polynomials. They explain this phenomenon in terms of the representation theory of finite abelian groups and the combinatorics of hyperplane arrangements. While their theory only applies when $n = 1$, Theorem \ref{thm neck eval intro} implies that the phenomenon of unexpected cyclotomic factors extends, albeit partially, to the higher necklace polynomials $M_{d,n}(x)$ as well.

Finally we note the generating function identity which plays a critical role in our analysis and may be of independent interest. Let $P_{d,n}(x) \in \QQ[x]$ denote the polynomial (see Section \ref{sec neck poly}) such that for all prime powers $q$,
\[
    |\poly_{d,n}(\FF_q)| = P_{d,n}(q).
\]

\begin{thm}[Higher Cyclotomic Identity]
\label{thm higher cyclo intro}
For each $n\geq 1$, the following identity holds in $\QQ[x][\![t]\!]$,
\[
    \sum_{d\geq 0}P_{d,n}(x)t^d = \prod_{j \geq 1}\Big(\frac{1}{1 - t^j}\Big)^{M_{j,n}(x)}.
\]
\end{thm}

When $n = 1$, $P_{d,1}(x) = x^d$ and Theorem \ref{thm higher cyclo intro} specializes to the classic so-called \textit{cyclotomic identity},
\[
    \frac{1}{1 - xt} = \sum_{d\geq 0} x^dt^d = \prod_{d\geq 1}\Big(\frac{1}{1 - t^d}\Big)^{M_d(x)}.
\]
There are several interpretations of the cyclotomic identity, including with $x = q$ as an Euler product for the Hasse-Weil zeta function $\zeta_{\AA^1(\FF_q)}(t)$ of the affine line.

\subsection{Organization}
The body of this paper is divided into four sections. In Section \ref{sec neck poly} we prove the higher cyclotomic identity. In Section \ref{sec euler} we realize the compactly supported Euler characteristics of $\irr_{d,n}(\RR)$ and $\irr_{d,n}(\CC)$ as values of the higher necklace polynomials. In Section \ref{sec neck eval} we evaluate the higher necklace polynomials at prime order $p$ roots of unity in terms of the balanced base $p$ expansion of the number of variables $n$. Finally, in Section \ref{sec conclusion} we bring everything together to express the compactly supported Euler characteristic of $\irr_{d,n}(\RR)$ in terms of the balanced binary expansion of $n$.

\subsection{Note on previous versions}
The results in this paper originally appeared in the author's Ph.D. dissertation \cite{hyde_thesis} and as the second half of the preprint \cite{hyde}. This earlier preprint has split into two papers with a coauthor added to the first half \cite{hyde odesky}.

\subsection{Acknowledgements}
We would like to thank Benson Farb, Nir Gadish, and Phil Tosteson for helpful discussions related to this work. The author is partially supported by the NSF MSPRF and the Jump Trading Mathlab Research Fund.

\section{Higher necklace polynomials}
\label{sec neck poly}
In this section we prove the higher cyclotomic identity. First we recall the definitions of the sets $\poly_{d,n}(K)$, $\irr_{d,n}(K)$ and their associated polynomials $P_{d,n}(x)$, $M_{d,n}(x)$ from \cite{hyde_liminal}. Let $d, n \geq 1$ be integers and let $K$ be an arbitrary field.

\begin{defn}
\label{def poly}
Let $\poly_{d,n}(K)$ be the set of $K^\times$-orbits of total degree $d$ polynomials in $K[x_1, x_2,\ldots, x_n]$. The \textit{total degree} of a monomial $x_1^{m_1}x_2^{m_2}\cdots x_n^{m_n}$ is $m_1 + m_2 + \ldots + m_n$ and the total degree of a polynomial $f \in K[x_1, x_2, \ldots, x_n]$ is the maximum total degree of the monomial summands of $f$. We refer to the elements of $\poly_{d,n}(K)$ as \textit{monic polynomials}. Let $\irr_{d,n}(K) \subseteq \poly_{d,n}(K)$ be the subset of all monic polynomials which are irreducible over $K$. 
\end{defn}

Note that if we choose a monomial ordering on $K[x_1, x_2, \ldots, x_n]$, then there is a unique representative for each $K^\times$-orbit with leading coefficient 1. Hence our notion of a monic polynomial is equivalent to the more common notion after such a choice of monomial order.

Recall that for any field $K$, the polynomial ring $K[x_1, x_2, \ldots, x_n]$ is a unique factorization domain. Unique factorization gives us the decomposition
\begin{equation}
\label{eqn unique fac}
    \poly_{d,n}(K) = \bigsqcup_{\lambda \vdash d}\prod_{j\geq 1}\Sym^{m_j(\lambda)}(\irr_{j,n}(K)),
\end{equation}
where the disjoint union is over all partitions $\lambda$ of $d$ and $m_j(\lambda)$ is the number of parts of $\lambda$ of size $j$. Here $\Sym^m(X) := X^m / S_m$ denotes the $m$-fold symmetric product.

In \cite[Lem. 2.1]{hyde_liminal} the author proved that there are polynomials $P_{d,n}(x), M_{d,n}(x) \in \QQ[x]$ such that for any finite field $\FF_q$,
\[
    |\poly_{d,n}(\FF_q)| = P_{d,n}(q)\hspace{.75in} |\irr_{d,n}(\FF_q)| = M_{d,n}(q).
\]
We call $M_{d,n}(x)$ the \textit{higher necklace polynomials}. Let $\multi{x}{m} \in \QQ[x]$ be the polynomial
\begin{equation}
\label{eqn multi def}
    \multi{x}{m} := \frac{x(x+1)(x+2)\cdots(x+m-1)}{m!} = \binom{x + m - 1}{m}.
\end{equation}
If $n \in \NN$, then $\multi{n}{m}$ may be interpreted as the number of ways to choose $m$ objects from a collection of $n$ objects with repetitions. Taking cardinalities of both sides in \eqref{eqn unique fac} with $K = \FF_q$ gives us the polynomial identity
\begin{equation}
\label{eqn unique fac polys}
    P_{d,n}(x) = \sum_{\lambda \vdash d} \prod_{j\geq 1} \multi{M_{j,n}(x)}{m_j(\lambda)}.
\end{equation}

\begin{defn}
\label{def bin ring}
A commutative ring $R$ is called a \textit{binomial ring} if
\begin{enumerate}
    \item $R$ is torsion free as an abelian group ($ma = 0$ with $m \in \ZZ$ and $a \in A$ implies $m = 0$ or $a = 0$,)
    \item For each $a \in R$ and $n\geq 0$, $\binom{a}{n} = \frac{1}{n!}a(a-1)(a-2)\cdots(a - n + 1) \in R$.
\end{enumerate}
\end{defn}

Binomial rings were defined by Philip Hall \cite{Hall} in his study of nilpotent groups. See Elliott \cite{Elliott} for an overview and further references on binomial rings. Examples of binomial rings include any localization of $\ZZ$, any $\QQ$-algebra, and the ring of integer valued polynomials in $\QQ[x]$. The second condition in Definition \ref{def bin ring} is equivalent to $\multi{a}{n} \in R$ for all $a \in R$ and $n\geq 0$ as can be seen by \eqref{eqn multi def}.

Let $R$ be a binomial ring and let $\Lambda(R) := 1 + t R[\![t]\!]$ be the multiplicative group of unital formal power series with coefficients in $R$. We use $\multi{x}{n}$ to define an exponential action of $R$ on certain elements of $\Lambda(R)$. In particular,
\[
    \left(\frac{1}{1 - t}\right)^a := \sum_{n\geq 0}\multi{a}{n}t^n.
\]
This identity is equivalent to the binomial theorem by \eqref{eqn multi def}.

Lemma \ref{lem comb euler} is well-known in the context of formal power series, symmetric functions, and the theory of Witt vectors but is typically not stated in the generality which we technically require. As we will make use of this several times we prove it here for completeness.

\begin{lemma}
\label{lem comb euler}
For any binomial ring $R$ and any sequence $a_d \in R$ for $d \geq 0$ such that $a_0 = 1$ there exists a unique sequence $b_j \in R$ for $j\geq 1$ such that the following identity holds in $\Lambda(R)$,
\begin{equation}
\label{eqn comb euler}
    \sum_{d\geq 0}a_d t^d = \prod_{j\geq 1}\left(\frac{1}{1 - t^j}\right)^{b_j}.
\end{equation}
Furthermore \eqref{eqn comb euler} is equivalent to
\[
    a_d = \sum_{\lambda \vdash d}b_\lambda
\]
for all $d\geq 1$ where for a partition $\lambda = (1^{m_1}2^{m_2}\cdots)$
\begin{equation}
\label{eq def b_lambda}
    b_\lambda := \prod_{j\geq 1}\multi{b_j}{m_j}.
\end{equation}
\end{lemma}

\begin{proof}
The right hand side of \eqref{eqn comb euler} expands as
\[
    \prod_{j\geq 1}\left(\frac{1}{1 - t^j}\right)^{b_j} = \prod_{j\geq 1}\sum_{m\geq 0}\multi{b_j}{m}t^{mj} = \sum_{d\geq 0}\sum_{\lambda\vdash d}b_\lambda t^d.
\]
We show by induction on $d$ that there exists a uniquely determined sequence $b_j$ such that for all $d\geq 1$,
\[
    a_d = \sum_{\lambda \vdash d}b_\lambda.
\]
For $d = 1$ there is only partition $\lambda$ and thus $a_1 = b_1$. Now suppose that $d > 1$ and that we have shown $b_j$ is uniquely determined for $j < d$. Then 
\[
    b_d = a_d - \sum_{\substack{\lambda \vdash d\\ \lambda \neq (d)}}b_\lambda.
\]
If $\lambda \neq (d)$, then all parts of $\lambda$ have size $j < d$ hence $b_d$ is uniquely determined by our induction hypothesis.
\end{proof}

\begin{thm}[Higher Cyclotomic Identity]
\label{thm higher cyclo}
For each $n\geq 1$, the following identity holds in $\Lambda(\QQ[x])$,
\[
    \sum_{d\geq 0}P_{d,n}(x)t^d = \prod_{j \geq 1}\Big(\frac{1}{1 - t^j}\Big)^{M_{j,n}(x)}.
\]
\end{thm}

\begin{proof}
This follows directly from Lemma \ref{lem comb euler} and the identity \eqref{eqn unique fac polys} with $R = \QQ[x]$, $a_d = P_{d,n}(x)$, and $b_j = M_{j,n}(x)$.
\end{proof}

\section{Euler characteristics as higher necklace values}
\label{sec euler}

The goal of this section is prove that the compactly supported Euler characteristics of $\irr_{d,n}(\RR)$ and $\irr_{d,n}(\CC)$ may be realized as values of the higher necklace polynomials. We begin with some background on compactly supported Euler characteristics.

\begin{defn}
\label{def euler char}
Say a topological space $X$ is \textit{tame} if the compactly supported singular cohomology $H_c^k(X,\QQ)$ (see Hatcher \cite[Pg. 243]{Hatcher}) vanishes for all but finitely many $k$. If $X$ is tame, then the \textit{compactly supported Euler characteristic} $\chi_c(X)$ is
\[
    \chi_c(X) := \sum_{k\geq 0}(-1)^k \dim_\QQ H_c^k(X, \QQ).
\]
\end{defn}

Affine and projective spaces over $\RR$ and $\CC$ are tame as are locally closed algebraic subsets of projective spaces and their images under algebraic maps. Tame spaces are closed under disjoint unions and Cartesian products. Lemma \ref{lem poly dec} implies that $\poly_{d,n}(\RR)$ and $\poly_{d,n}(\CC)$ are tame. It then follows by induction from \eqref{eqn unique fac} that $\irr_{d,n}(\RR)$ and $\irr_{d,n}(\CC)$ are tame for all $d, n \geq 1$.

\begin{lemma}
\label{lem poly dec}
The set $\poly_{\leq d, n}(K)$ of all monic polynomials with total degree at most $d$ is naturally in bijection with the projective space $\PP^{\binom{d + n}{n}-1}(K)$, and thus with respect to the natural inclusion $\poly_{\leq d-1, n}(K) \subseteq \poly_{\leq d, n}(K)$ we can identify $\poly_{d,n}(K)$ with the complement
\[
    \poly_{d,n}(K) \cong \PP^{\binom{d + n}{n}-1}(K) \setminus \PP^{\binom{d + n - 1}{n}-1}(K).
\]
\end{lemma}

\begin{proof}
Consider the $K$-vector space spanned by all monomials in $n$ variables of degree at most $d$. A standard counting argument implies that this space has dimension $\binom{d + n}{n}$. The projectivization of this vector space is, by definition, the space of all non-zero monic degree at most $d$ polynomials in $K[x_1, x_2, \ldots, x_n]$. Hence $\poly_{\leq d, n}(K) \cong \PP^{\binom{d + n}{n}-1}(K)$.
\end{proof}

Lemma \ref{lem EC props} recalls some of the facts about $\chi_c$ that we will need.

\begin{lemma}
\label{lem EC props}
Suppose that $X$ and $Y$ are tame spaces. Then
\begin{enumerate}
    \item $\chi_c(X \sqcup Y) = \chi_c(X) +\chi_c(Y)$,
    \item $\chi_c(X\times Y) = \chi_c(X)\chi_c(Y)$,
    \item $\chi_c(\RR) = -1$ and $\chi_c(\CC) = 1$,
    \item If $K = \RR$ or $\CC$, then $\chi_c(\PP^{n-1}(K)) = [n]_{\chi_c(K)}$, where
    \[
        [n]_x := \frac{x^n - 1}{x - 1}.
    \]
    \item $\Sym^m X$ is tame for all $m\geq 1$ and
    \[
        \chi_c(\Sym^m X) = \multi{\chi_c(X)}{m}.
    \]
    Equivalently, in $\Lambda(\ZZ) = 1 + t\ZZ[\![t]\!]$ we have
    \[
        \sum_{d\geq 0}\chi_c(\Sym^d X)t^d = \left(\frac{1}{1 - t}\right)^{\chi_c(X)}.
    \]
\end{enumerate}
\end{lemma}

\begin{proof}
Part (1) is a general property of Euler characteristics and (2) holds for the compactly supported Euler characteristic (see \cite[Thm. 9.3.1]{spanier}.)  To see (3) we begin with the homeomorphism
\[
    \RR = (-\infty, 0) \sqcup \{0\} \sqcup (0,\infty) \cong \RR \sqcup \{0\} \sqcup \RR,
\]
and take $\chi_c$ of both sides to get
\[
    \chi_c(\RR) = 2\chi_c(\RR) + 1 \Longrightarrow \chi_c(\RR) = -1.
\]
Then $\CC \cong \RR^2$ and (2) imply $\chi_c(\CC) = \chi_c(\RR)^2 = 1$. To compute the Euler characteristic of projective space we use the cell decomposition
\[
    \PP^{n-1}(K) = K^{n-1} \sqcup K^{n-2} \sqcup \ldots \sqcup K \sqcup 1,
\]
where $1 := K^0$ is the one point space. Taking $\chi_c$ when $K = \RR$ or $\CC$ we have
\[
    \chi_c(\PP^{n-1}(K)) = \chi_c(K)^{n-1} + \chi_c(K)^{n-2} + \ldots + \chi_c(K) + 1 = [n]_{\chi_c(K)}.
\]
The final assertion (5) is a theorem due to MacDonald \cite{Macdonald}; see Vakil's notes \cite[Thm. 2.3]{Vakil Notes} for a nice one line proof.
\end{proof}

Lemma \ref{lem poly dec} and the cell decomposition of projective space used in the proof of Lemma \ref{lem EC props} together imply that
\begin{equation}
\label{eqn P exp}
    P_{d,n}(x) = \left[\binom{n+d}{n}\right]_{x} - \left[\binom{n+d-1}{n}\right]_{x}.
\end{equation}
We now prove the main result of this section.

\begin{thm}
\label{thm euler char}
If $d, n \geq 1$, then
\[
    \chi_c(\irr_{d,n}(\RR)) = M_{d,n}(-1) \hspace{.75in}
    \chi_c(\irr_{d,n}(\CC)) = M_{d,n}(1).
\]
\end{thm}

\begin{proof}
Let $K = \RR$ or $\CC$. Then \eqref{eqn unique fac} and Lemma \ref{lem EC props} imply that
\[
    \chi_c(\poly_{d,n}(K)) = \sum_{\lambda\vdash d}\prod_{j\geq 1}\chi_c(\Sym^{m_j}(\irr_{j,n}(K))) = \sum_{\lambda\vdash d}\prod_{j\geq 1} \multi{\chi_c(\irr_{j,n}(K))}{m_j}.
\]
Lemma \ref{lem comb euler} implies that the above identity is equivalent to
\[
    \sum_{d\geq 0}\chi_c(\poly_{d,n}(K))t^d = \prod_{j\geq 1}\left(\frac{1}{1 - t^j}\right)^{\chi_c(\irr_{j,n}(K))}.
\]
On the other hand, Lemma \ref{lem poly dec} and Lemma \ref{lem EC props}(4) show that
\begin{align*}
    \chi_c(\poly_{d,n}(K)) &= \chi_c(\PP^{\binom{n+d}{n}-1}(K)) - \chi_c(\PP^{\binom{n + d - 1}{n}-1}(K))\\
    &= \left[\binom{n+d}{n}\right]_{\chi_c(K)} - \left[\binom{n+d-1}{n}\right]_{\chi_c(K)}\\
    &= P_{d,n}(\chi_c(K)).
\end{align*}
The higher cyclotomic identity (Theorem \ref{thm higher cyclo}) implies that
\[
    \sum_{d\geq 0}P_{d,n}(\chi_c(K))t^d = \prod_{j\geq 1}\left(\frac{1}{1 - t^j}\right)^{M_{j,n}(\chi_c(K))}.
\]
Hence by the uniqueness of Lemma \ref{lem comb euler} we conclude that for all $d, n \geq 1$,
\[
    \chi_c(\irr_{d,n}(\RR)) = M_{d,n}(\chi_c(\RR)) = M_{d,n}(-1),
\]
and similarly for $K = \CC$.
\end{proof}

\section{Evaluating higher necklace polynomials at roots of unity}
\label{sec neck eval}

In this section we express the values of higher necklace polynomials $M_{d,n}(\zeta_p)$ at prime order roots of unity in terms of the balanced base $p$ expansion of the number of variables $n$.

\begin{defn}
\label{def balanced}
Let $b\geq 2$ and $n\geq 1$ be integers. A \textit{balanced base $b$ expansion of $n$} is an expression
\[
    n = b^{k_{2m}} - b^{k_{2m-1}} + b^{k_{2m-2}} - \ldots + b^{k_{2}} - b^{k_1},
\]
where $0\leq k_0 < k_1 < \ldots < k_{2m} $ is an increasing sequence of integers of even length and the coefficients of the powers of $b$ alternate between $\pm 1$.
\end{defn}

\begin{lemma}
\label{lem balanced}
Let $n \geq 0$ and $b \geq 2$ be integers. Then $n$ has a balanced base $b$ expansion if and only if the base $b$ expansion of $n$ only contains the digits $0$ and $b - 1$. Furthermore, if $n$ has a balanced base $b$ expansion, then it is unique.
\end{lemma}

\begin{proof}
First suppose that the base $b$ expansion of $n$ contains only the digits $0$ and $b - 1$,
\[
    n = (b - 1)b^{k_m} + (b - 1)b^{k_{m-1}} + \ldots + (b - 1)b^{k_1}.
\]
Expanding each $(b - 1)b^k = b^{k+1} - b^k$ term and collecting coefficients on each power of $b$ gives a balanced base $b$ expansion of $n$. Note that each nonzero base $b$ digit of $n$ contributes two terms to the balanced expansion and cancellation occurs in pairs, hence the expansion we get in this way has an even number of terms.

Next suppose that $n$ has a balanced base $b$ expansion
\begin{equation}
\label{eqn balanced}
    n = b^{k_{2m}} - b^{k_{2m-1}} + b^{k_{2m-2}} - \ldots + b^{k_2} - b^{k_1}.
\end{equation}
If $k_{2\ell} \neq k_{2\ell - 1} + 1$, then we can replace $b^{2\ell} - b^{2\ell-1}$ in \eqref{eqn balanced} with
\begin{align*}
    b^{k_{2\ell}} - b^{k_{2\ell-1}} &= b^{k_{2\ell}} - b^{k_{2\ell}-1} + b^{k_{2\ell}-2} - b^{k_{2\ell}-3} + \ldots + b^{k_{2\ell - 1}+1} - b^{k_{2\ell -1}}\\
    &= (b - 1)b^{k_{2\ell}-1} + (b - 1)b^{2\ell -3} + \ldots + (b - 1)b^{2\ell - 1}.
\end{align*}
Note that here we are using that a balanced base $b$ expansion has an even number of terms. Hence, after making this replacement for each consecutive pair of terms in \eqref{eqn balanced}, we have a base $b$ expansion for $n$ which only contains the digits $0$ and $b - 1$.

The uniqueness of balanced base $b$ expansions then follows from the above argument and the uniqueness of the usual base $b$ expansion of $n$.
\end{proof}

\begin{example}
Every positive integer has a balanced base $2$ expansion since $0$ and $b -1 = 1$ are the only possible binary digits. For example the balanced base $2$ expansion of $n = 55$ is
\[
    55 = 2^6 - 2^4 + 2^3 - 1
\]
\end{example}

\begin{thm}
\label{thm balanced}
Let $p$ be a prime and let $n\geq 1$ be an integer such that
\[
    n = \sum_{k\geq 0}b_kp^k
\]
is the balanced base $p$ expansion of $n$. If $\zeta_p$ is a primitive $p$th root of unity, then
\[
    M_{d,n}(\zeta_p) = \begin{cases} b_k & \text{if }d = p^k,\\ 0 & \text{otherwise.}\end{cases}
\]
Thus it follows that $\Phi_p(x)$ divides $M_{d,n}(x)$ for all but finitely many $d\geq 1$ whenever $n$ has a balanced base $p$ expansion.
\end{thm}

Before proving Theorem \ref{thm balanced} we prove two lemmas. Recall that if $m \geq 0$ is an integer we write
\[
    [m]_x := \frac{x^m - 1}{x - 1} = x^{m-1} + x^{m-2} + \ldots + x + 1.
\]
\begin{lemma}
\label{lem qint}
If $\zeta$ is a non-trivial $n$th root of unity, then $[m]_\zeta$ depends only on $m$ modulo $n$.
\end{lemma}

\begin{proof}
If $\zeta$ is a nontrivial $n$th root of unity, then
\[
    [n]_\zeta = \zeta^{n-1} + \zeta^{n-2} + \ldots + \zeta + 1 = 0.
\]
If $m = an + b$, then
\[
    [m]_x = \frac{x^{an + b} - 1}{x - 1}
    = x^b\cdot\frac{x^{an} - 1}{x^n - 1}\cdot\frac{x^n - 1}{x - 1} + \frac{x^b - 1}{x - 1}
    = x^b[a]_{x^n}[n]_x + [b]_x.
\]
Evaluating at $x = \zeta$ gives
\[
    [m]_\zeta = \zeta^b a [n]_\zeta + [b]_\zeta = [b]_\zeta.\qedhere
\]
\end{proof}

Lemma \ref{lem Lucas} is a beautiful result due to Lucas \cite{Lucas}; see Fine \cite{Fine} for a modern proof.

\begin{lemma}[Lucas's Theorem]
\label{lem Lucas}
If $p$ is a prime and
\begin{align*}
    m &= a_k p^k + a_{k-1}p^{k-1} +\ldots + a_1p + a_0\\
    n &= \hspace{1pt}b_k p^k + \hspace{1pt}b_{k-1}p^{k-1} +\ldots + b_1p \hspace{1pt}+ \hspace{1pt}b_0
\end{align*}
are the base $p$ expansions of the natural numbers $m$ and $n$ (possibly with leading zero coefficients), then
\[
    \binom{m}{n} \equiv \binom{a_k}{b_k}\binom{a_{k-1}}{b_{k-1}}\cdots \binom{a_1}{b_1}\binom{a_0}{b_0} \bmod p.
\]
\end{lemma}

We now prove Theorem \ref{thm balanced}.

\begin{proof}[Proof of Theorem \ref{thm balanced}]
Suppose that $n$ has a balanced base $p$ expansion and let $\zeta_p$ be a non-trivial $p$th root of unity. Then by Theorem \ref{thm higher cyclo},
\begin{equation}
\label{eqn gen cyc eval}
    \sum_{d\geq 0}P_{d,n}(\zeta_p)t^d = \prod_{j\geq 1}\left(\frac{1}{1 - t^j}\right)^{M_{j,n}(\zeta_p)}.
\end{equation}
We evaluate $M_{d,n}(\zeta_p)$ by expressing the left hand side of \eqref{eqn gen cyc eval} as an infinite product of the same form in another way and then appealing to the uniqueness of Lemma \ref{lem comb euler}. Towards that end, let $Q(t) \in \Lambda(\QQ(\zeta_p))$ be defined by
\[
    Q(t) := \sum_{d\geq 0} \left[\binom{d+n}{n}\right]_{\zeta_p} t^d.
\]
Lemma \ref{lem poly dec} implies that $P_{d,n}(x)$ may be expressed as
\begin{equation}
\label{eqn q expr}
    P_{d,n}(x) = \frac{x^{\binom{d+n}{n}} - x^{\binom{d+n-1}{n}}}{x - 1} = \left[\binom{d+n}{n}\right]_x - \left[\binom{d+n-1}{n}\right]_x.
\end{equation}
Thus,
\begin{align*}
    \sum_{d\geq 0}P_{d,n}(\zeta_p)t^d &= \sum_{d\geq 0}\left(\left[\binom{d+n}{n}\right]_{\zeta_p} - \left[\binom{d+n-1}{n}\right]_{\zeta_p}\right)t^d \\
    &= \sum_{d\geq 0}\left[\binom{d+n}{n}\right]_{\zeta_p} t^d - t\sum_{d\geq 1}\left[\binom{d+n-1}{n}\right]_{\zeta_p} t^{d-1}\\
    &= Q(t) - tQ(t)\\
    &= (1 - t)Q(t).
\end{align*}
Next we determine the coefficients of $Q(t)$. Say positive integers $d$ and $n$ are \emph{$p$-complementary} if there is no $p^k$ with a non-zero coefficient in the base $p$ expansions of both $d$ and $n$. Now suppose that $n$ has a balanced base $p$ expansion. If $d$ and $n$ are not $p$-complementary, suppose $p^k$ is the smallest power of $p$ common to the base $p$ expansions of $d$ and $n$. Then the coefficient of $p^k$ in $d + n$ is 0 since
\begin{enumerate}
    \item the coefficient of $p^k$ in $n$ is $p - 1$ by our assumption that $n$ has a balanced base $p$ expansion,
    \item the coefficient of $p^k$ in $d$ is at least 1, and
    \item the minimality of $k$ implies there are no carries for smaller power $p$ in the sum $d + n$.
\end{enumerate}
Thus Lucas's theorem (Lemma \ref{lem Lucas}) implies that if $d$ and $n$ are not $p$-complementary, then
\[
    \binom{d+n}{n} \equiv 0 \bmod p
\]
since the factor corresponding to $p^k$ will be 0. Therefore, if $d$ and $n$ are not $p$-complementary, then by Lemma \ref{lem qint} we have
\[
    \left[\binom{d+n}{n}\right]_{\zeta_p} = [0]_{\zeta_p} = 0.
\]

Next suppose $d$ and $n$ have base $p$ expansions
\begin{align*}
    d &= a_kp^k + a_{k-1}p^{k-1} + \ldots + a_1 p + a_0\\
    n &= b_kp^k + b_{k-1}p^{k-1} + \ldots + b_1p + b_0.
\end{align*}
If $d$ and $n$ are $p$-complementary, then the base $p$ expansion of $d + n$ is
\[
    d+n = (a_k + b_k)p^k + (a_{k-1} + b_{k-1})p^{k-1} + \ldots + (a_1 + b_1)p + (a_0 + b_0)
\]
where for each $i$ at most one of $a_i$ and $b_i$ is non-zero. Lucas's theorem implies that
\[
    \binom{d+n}{n} \equiv \binom{a_k + b_k}{b_k}\binom{a_{k-1} + b_{k-1}}{b_{k-1}}\cdots \binom{a_1 + b_1}{b_1}\binom{a_0 + b_0}{b_0} \bmod p.
\]
Let $0 \leq i \leq k$. If $a_i = 0$, then
\[
    \binom{a_i + b_i}{b_i} = \binom{b_i}{b_i} = 1,
\]
and if $b_i = 0$ then 
\[
    \binom{a_i + b_i}{b_i} = \binom{a_i}{0} = 1.
\]
Therefore, if $d$ and $n$ are $p$-complementary, then
\[
    \binom{d+n}{n} \equiv 1 \bmod p.
\]
Hence by Lemma \ref{lem qint},
\[
    \left[\binom{d+n}{n}\right]_{\zeta_p} = [1]_{\zeta_p} = 1.
\]
Combining these computations we have
\[
    Q(t) = \sum_{d\geq 0} \left[\binom{d+n}{n}\right]_{\zeta_p} t^d = \sum_{\substack{d \text{ is $p$-comp.}\\ \text{to }n}} t^d.
\]
The existence and uniqueness of base $p$ expansions of natural numbers is equivalent to the following product formula,
\[
    \frac{1}{1 - t} = \sum_{d\geq 0}t^d = \prod_{k\geq 1}\sum_{a=0}^{p-1}t^{ap^k} = \prod_{k\geq 1} \frac{1 - t^{p^{k+1}}}{1 - t^{p^k}},
\]
where the factor of $\frac{1 - t^{p^{k+1}}}{1 - t^{p^k}}$ contributes to the coefficient of $t^d$ precisely when $d$ is not $p$-complementary to $p^k$. If $n = (p-1)p^{\ell_1} + (p-1)p^{\ell_2} + \ldots + (p-1)p^{\ell_s}$ is the base $p$ expansion of $n$ (which can be expressed in this form by the assumption that $n$ has a balanced base $p$ expansion,) then 
\[
    Q(t) = \sum_{\substack{d \text{ is $p$-comp.}\\ \text{to }n}} t^d = \frac{1}{1-t}\prod_{i=1}^s \frac{1 - t^{p^{\ell_i}}}{1 - t^{p^{\ell_i+1}}}.
\]
Therefore
\[
    \sum_{d\geq 0}P_{d,n}(\zeta_p)t^d = (1-t)Q(t) = \prod_{i=1}^s \frac{1 - t^{p^{\ell_i}}}{1 - t^{p^{\ell_i+1}}} = \prod_{k\geq 1}\left(\frac{1}{1 - t^{p^k}}\right)^{b_k},
\]
where $n = \sum_{k\geq 0}b_k p^k$ is the balanced base $p$ expansion of $n$. The uniqueness of product expansions of this form provided by Lemma \ref{lem comb euler} implies that $M_{p^k,n}(\zeta_p) = b_k$ and $M_{d,n}(\zeta_p) = 0$ when $d$ is not a power of $p$.
\end{proof}

For a fixed $n$ there are finitely many primes $p$ for which $n$ has a balanced base $p$ expansion. Theorem \ref{thm balanced} tells us that for each such prime $p$ there are only finitely many $d$ such that $M_{d,n}(\zeta_p) \neq 0$ for $\zeta_p$ a primitive $p$th root of unity.

\begin{remark}
When $n$ does not have a balanced base $p$ expansion, one can still use the methods of Theorem \ref{thm balanced} to determine the values of $M_{d,n}(\zeta_p)$, but the results become more challenging to state concisely. Since our main interest is in the case $p = 2$ and its application to evaluating Euler characteristics, we defer this more refined analysis.
\end{remark}

We end this section with an evaluation of $M_{d,n}(1)$ which will be used below to compute the compactly supported Euler characteristics of $\irr_{d,n}(\CC)$.

\begin{prop}
\label{prop eval 1}
For all $d, n \geq 1$,
\[
    M_{d,n}(1) = \begin{cases}n & \text{if } d = 1,\\ 0 & \text{otherwise.}\end{cases}
\]
\end{prop}

\begin{proof}
Note that $[m]_1 = m$, hence by \eqref{eqn q expr} we have
\[
    P_{d,n}(1) = \binom{d + n}{n} - \binom{d + n - 1}{n} = \binom{d + n - 1}{d} = \multi{n}{d}.
\]
Therefore
\[
    \sum_{d\geq 0}P_{d,n}(1)t^d = \sum_{d\geq 0}\multi{n}{d}t^d = \left(\frac{1}{1 - t}\right)^n.
\]
Thus the higher cyclotomic identity and the uniqueness of Lemma \ref{lem comb euler} imply that $M_{1,n}(1) = n$ and $M_{d,n}(1) = 0$ for $d > 1$.
\end{proof}

\section{Conclusion}
\label{sec conclusion}

Combining the results of the previous sections we now arrive at the main result.

\begin{thm}
\label{thm main}
Let $d, n \geq 1$ and suppose that $n = \sum_{k=0}^\ell b_k 2^k$ is the balanced binary expansion of $n$. Then
\[
    \chi_c(\irr_{d,n}(\RR)) = \begin{cases} b_k & \text{if }d = 2^k,\\ 0 & \text{otherwise.}\end{cases}
\]
\end{thm}

\begin{proof}
Theorem \ref{thm euler char} implies that $\chi_c(\irr_{d,n}(\RR)) = M_{d,n}(-1)$ and Theorem \ref{thm balanced} expresses the value of $M_{d,n}(-1)$ in terms of the balanced binary expansion of $n$.
\end{proof}

Similarly Proposition \ref{prop eval 1} gives the following computation of the Euler characteristic of $\irr_{d,n}(\CC)$.

\begin{prop}
Let $d, n \geq 1$. Then 
\[
    \chi_c(\irr_{d,n}(\CC)) = \begin{cases} n & \text{if }d = 1,\\ 0 & \text{otherwise.}\end{cases}
\]
\end{prop}

As noted in the introduction, Theorem \ref{thm main} suggests that the space $\irr_{d,n}(\RR)$ may have an interesting cell decomposition determined by the binary expansion of the number of variables $n$. 

\end{document}